\documentclass{amsart}
\usepackage{amsmath, amssymb, mathdots}
\newtheorem{Lem}{Lemma}
\newtheorem{Theo}[Lem]{Theorem}
\newtheorem{Prop}[Lem]{Proposition}
\newtheorem{Cor}[Lem]{Corollary}
\newtheorem{Prob}{Problem}
\newtheorem{Conj}[Prob]{Conjecture}

\def\N{{\mathbb N}}

\begin{document}
\title[Riemann Zeta Function and the fractional part ]{On the Riemann Zeta Function and the fractional part of rational powers}
\author[T. Barnea]{Tal Barnea}
\address{}
\email{tal.q.barnea@gmail.com}

\maketitle
\begin{abstract}
Using elementary methods we find surprising connections between the values of the Riemann Zeta Function over integers and the fractional parts of rational powers, and a connection between the Riemann Zeta Function and the Prime Zeta Function.
\end{abstract}

\section{Introduction}
Given a real number $x$ we recall that the \textbf{simple continued fraction} of it is
$$x=a_0+\cfrac{1}{a_1+\cfrac{1}{a_2+\cfrac{1}{a_3+\cdots}}}$$  where all $a_i$ are integers and $a_i  \ge  1$ for $i \ge 1$. Thus, the first term $a_0=\lfloor x  \rfloor$ and the second term, is $a_1=\left\lfloor \frac{1}{x-a_0 }  \right\rfloor$. 

The Riemann Zeta function is defined for real $s>1$ as $$\zeta(s)=\sum_{n=1}^{\infty} \frac{1}{n^s}.$$
For all positive integers $n \ge 2$ we have that $1 < \zeta(n)<2$.
Therefore, the first term of $\zeta(n)$ in its simple continued fraction is always $1$. 
In the OEIS sequence A013697, Second term in continued fraction for zeta(n), it is stated by F.~Adams-Watters: "It appears that $a(n) = 2^n - \left\lfloor \left( \frac{4}{3} \right)^n \right\rfloor - k$, where $k$ is usually $2$, but is sometimes $1$. Up to $n=1000$, the only values of $n$ where $k = 1$ are $4$, $5$, $13$, $14$, and $17$. That is, $$\left\lfloor \frac{1}{\zeta(n)-1} \right\rfloor=2^n-\left\lfloor \left( \frac{4}{3} \right)^n \right\rfloor-k,$$ where $k=1$ or $k=2$".  
We first prove that this formula holds for all $n \geq 2$.

\begin{Theo}  \label{Theo:Main}
For all natural numbers $n \ge 2$, we have that $$\left\lfloor \frac{1}{\zeta(n)-1} \right\rfloor=2^n-\left\lfloor \left( \frac{4}{3} \right)^n \right\rfloor-k,$$    where $k=2$ except for finite exceptions where $k=1$.
\end{Theo}

We write $ \{ x \}=x-\lfloor x \rfloor$ , for the fractional part of $x$. Given coprime integers  $p>q>1$, we consider the set $\left\{ \left\{ \left( \frac{p}{q} \right)^{n} \right\} \mid n \in \N \right\}$. T.~Vijayaraghavan proved in \cite{Vi} that this set has infinitely many limit points, but otherwise not much is known about its distribution.  As a corollary of Theorem~\ref{Theo:Main} we obtain the following surprising result. For real $s$ write $\varepsilon_x(s)=(2x)^s \left( \left( \frac{2}{3} \right)^s + \left( \frac{1}{2} \right)^s \right)^2$ and for simplicity we write $\varepsilon(s)=\varepsilon_1(s)$.

\begin{Cor} \label{Cor:FractionalParts} 
For all natural numbers, except finite exceptions, $n \ge 2$, we have that $$1< \left\{  \frac{1}{\zeta(n)-1} \right\} + \left\{ \left( \frac{4}{3} \right)^n \right\}<1+\varepsilon(n).$$
Therefore, the set $\left\{  \left\{ \frac{1}{\zeta(n)-1}  \right\}  \mid n \in \N \right\}$ has infinitely many limit points.
\end{Cor}

\begin{Cor} \label{Cor:Egypt}
If $k=2$ in Theorem~\ref{Theo:Main}, then $\zeta(n)$ cannot be of the form $1+ \frac{1}{m}$ where $m$ is a natural number.
\end{Cor}

\begin{Conj} \label{Conj1}
$k=1$ only when $n=4,5,13,14,17$.
\end{Conj} 
Using Mathematica we checked when $\left\{ \left( \frac{4}{3} \right)^n \right\}$ is less than $10^{-9}$ for $n \le 5,000,000$, and no examples were found. As $\varepsilon(n)<10^{-9}$ for $n \ge 176$, Corollary~\ref{Cor:FractionalParts} implies that the conjecture holds for $176 \leq n \leq 5,000,000$. Since F.~Adams-Watters verified the conjecture for $n \leq 1000$, we conclude the conjecture holds for $n \le 5,000,000$.

\begin{Theo} \label{Theo:Zeta}
For $n$ large enough we have
$$1-\left( \frac{8}{9} \right)^n- \left( \frac{2}{3} \right)^n < \left\{  \frac{1}{\zeta(n)-1} \right\} + \left\{  \frac{\left( \frac{2}{3} \right)^n}{\zeta(n)-1} \right\} < $$
$$1-\left( \frac{8}{9} \right)^n- \left( \frac{2}{3} \right)^n +\varepsilon(n) + \varepsilon_{\frac{2}{3}}(n)$$
except for finite exceptions where either
$$0 < \left\{  \frac{1}{\zeta(n)-1} \right\} + \left\{  \frac{\left( \frac{2}{3} \right)^n}{\zeta(n)-1} \right\} <\varepsilon(n) + \varepsilon_{\frac{2}{3}}(n)-\left( \frac{8}{9} \right)^n- \left( \frac{2}{3} \right)^n$$
or
$$2-\left( \frac{8}{9} \right)^n- \left( \frac{2}{3} \right)^n < \left\{  \frac{1}{\zeta(n)-1} \right\} + \left\{  \frac{\left( \frac{2}{3} \right)^n}{\zeta(n)-1} \right\} < 2$$
holds.
\end{Theo}

Let $P(s)$ be the Prime Zeta Function, that is, $P(s)=\sum\limits_{m=1}^{\infty} \frac{1}{(p_m)^s}$ where $p_m$ is the $m^{th}$ prime. For real $s$ let $\delta(s)=2^s \left( \left( \frac{2}{3} \right)^s+\left( \frac{2}{5} \right)^s \right)^2 -\left( \frac{4}{5} \right)^s$.
\begin{Theo} \label{Theo:Prime}
For all real $s \ge 7$ we have that $$1-\varepsilon(s)<\frac{1}{P(s)} - \frac{1}{\zeta(s)-1} < 1+\delta(s)$$
\end{Theo}
As the dominating terms of both $P(s)$ and $\zeta(s)-1$ are the same one might expect their their reciprocals would tend to each other as $s$ tends to infinity. Thus, the somewhat surprising fact is that their difference is bounded away from $0$.
 
\vspace{.5cm}
\noindent \textbf{Acknowledgements.} 
I would like thank James McKee for his alternative proof for Proposition~\ref{Prop:Upper} using a geometric expansion, this improvement turned out to be crucial for this paper. I would also like to thank Jan-Christoph Schlage-Puchta for his helpful comments and discussions on previous drafts of this paper, as well as writing a Mathematica programme. In addition, I thank Vsevolod Lev helpful suggestion on a previous draft and for Yann Bugeaud for pointing out Malher's Theorem which made this paper much stronger. Finally I would like to thank my father, Yiftach Barnea, for helping me understand the previous literature, and with writing and organising this paper.

\section{Analytic Results}

\begin{Prop} \label{Prop:Lower}
For all $n \ge 2$ we have that $$1<\frac{1}{\zeta(n)-1}-2^n+\left( \frac{4}{3} \right)^n+2.$$
\end{Prop}
\begin{proof}
We prove this by contradiction. Assume $$1\geq \frac{1}{\zeta(n)-1}-2^n+\left( \frac{4}{3} \right)^n+2.$$
We will first cancel out some terms from our inequality.
Our assumption implies that  $$2^n-\left( \frac{4}{3} \right)^n-1 \geq \frac{1}{\zeta(n)-1},$$
$$\frac{6^n-4^n-3^n}{3^n} \geq \frac{1}{\zeta(n)-1},$$
$$\frac{3^n}{6^n-4^n-3^n} \leq \zeta(n)-1,$$
$$3^n \leq \left(6^n-4^n-3^n \right) \sum_{i=2}^{\infty} \frac{1}{i^n},$$
$$3^n+ \sum_{i=2}^{\infty} \frac{3^n}{i^n} \leq \sum_{i=2}^{\infty} \frac{6^n-4^n}{i^n},$$
$$3^n+ \sum_{i=2}^{\infty} \frac{3^n}{i^n} \leq 3^n-2^n +\sum_{i=3}^{\infty} \frac{6^n-4^n}{i^n},$$
$$2^n+ \sum_{i=2}^{\infty} \frac{3^n}{i^n} \leq \sum_{i=3}^{\infty} \frac{6^n-4^n}{i^n},$$
$$2^n- \frac{6^n-4^n}{3^n}+ \sum_{i=2}^{\infty} \frac{3^n}{i^n} \leq \sum_{i=4}^{\infty} \frac{6^n-4^n}{i^n},$$
$$\frac{4^n}{3^n} + \sum_{i=2}^{\infty} \frac{3^n}{i^n} \leq \sum_{i=4}^{\infty} \frac{6^n-4^n}{i^n},$$
$$\frac{3^n}{2^n} + \frac{4^n}{3^n} + \sum_{i=3}^{\infty} \frac{3^n}{i^n} \leq \frac{6^n-4^n}{4^n} + \sum_{i=5}^{\infty} \frac{6^n-4^n}{i^n},$$
$$2+\frac{4^n}{3^n} +\frac{3^n}{4^n}+ \sum_{i=5}^{\infty} \frac{3^n}{i^n}=1 + \frac{4^n}{3^n} + \sum_{i=3}^{\infty} \frac{3^n}{i^n} \leq \sum_{i=5}^{\infty} \frac{6^n-4^n}{i^n},$$
\begin{equation}
\label{eq:2}
2+\frac{16^n+9^n}{12^n}-\frac{6^n-4^n-3^n}{5^n} \leq \sum_{i=6}^{\infty} \frac{6^n-4^n-3^n}{i^n}.
\end{equation}
It is clear that $$\sum_{i=6}^{\infty} \frac{6^n-4^n-3^n}{i^n} < \int_{5}^{\infty} \frac{6^n-4^n-3^n}{t^n} dt$$
Substituting this into \eqref{eq:2} we get $$2+\frac{16^n+9^n}{12^n}-\frac{6^n-4^n-3^n}{5^n}<\int_{5}^{\infty} \frac{6^n-4^n-3^n}{t^n} dt=\left(6^n-4^n-3^n \right) \left( \frac{5^{1-n}}{n-1} \right),$$
$$2+\frac{16^n+9^n}{12^n}<\left(6^n-4^n-3^n \right) \left( \frac{1}{5^n} +\frac{1}{5^{n-1}(n-1)} \right)=\left(6^n-4^n-3^n \right) \left( \frac{n+4}{5^n(n-1)} \right),$$
$$\left( \frac{5^n}{6^n-4^n-3^n} \right) \left( \frac{16^n+9^n}{12^n} +2 \right) < \frac{n+4}{n-1}.$$
Let $$f(x)=\left( \frac{5^x}{6^x-4^x-3^x} \right) \left( \frac{16^x+9^x}{12^x} +2 \right)- \frac{x+4}{x-1}.$$
We deduce from our assumption that for $n \ge 2$ we have that $f(n)<0$.

For $n=2,3,4,5,6$ it can be checked that the above statement is false and therefore the assumption does not hold for these values. We will now prove that $f(x)>0$ for $x \geq 7$.
It is clear that $$f(x) > \left( \frac{5^x}{6^x} \right) \left( \frac{16^x}{12^x} \right) - \frac{x+4}{x-1}= \left( \frac{10^x}{9^x} \right) -\frac{x+4}{x-1}>0$$
as $\left( \frac{10^7}{9^7} \right) -\frac{7+4}{7-1}>0$ and $\left( \frac{10^x}{9^x} \right) -\frac{x+4}{x-1}$ has a positive derivative for all $x$. 
This is a contradiction and thus, $$1< \frac{1}{\zeta(n)-1}-2^n+\left( \frac{4}{3} \right)^n+2$$ for all $n \ge 2$.
\end{proof}

\begin{Prop} \label{Prop:Upper}
For $n \ge 2$ we have that $$\frac{1}{\zeta(n)-1}-2^n+\left( \frac{4}{3} \right)^n+2 < 1+ \varepsilon(n).$$
\end{Prop}
\begin{proof}
Consider:$$\frac{1}{\zeta(n)-1}=\frac{1}{\frac{1}{2^n}+\frac{1}{3^n}+\frac{1}{4^n}+\frac{1}{5^n}+\cdots} < \frac{1}{\frac{1}{2^n}+\frac{1}{3^n}+\frac{1}{4^n}}$$

$$\frac{1}{\zeta(n)-1} < \frac{2^n}{1+\frac{2^n}{3^n} +\frac{1}{2^n}}=$$
$$2^n \left(1-\left( \frac{2^n}{3^n}+ \frac{1}{2^n} \right) +\left( \frac{2^n}{3^n}+ \frac{1}{2^n} \right)^2 - \left( \frac{2^n}{3^n}+ \frac{1}{2^n} \right)^3 + \cdots \right)<$$
$$2^n \left(1-\left( \frac{2^n}{3^n}+ \frac{1}{2^n} \right) +\left( \frac{2^n}{3^n}+ \frac{1}{2^n} \right)^2 \right)=2^n -\left( \frac{4}{3} \right)^n -1+2^n \left( \frac{2^n}{3^n}+ \frac{1}{2^n} \right)^2.$$
We conclude that $$\frac{1}{\zeta(n)-1}-2^n+\left( \frac{4}{3} \right)^n+2 < 1+ \varepsilon(n),$$ for $n \ge 2$.
\end{proof}

\begin{Prop} \label{Prop:PrimeUpper}
For $s > 1$ we have$$\frac{1}{P(s)}-2^s+\left( \frac{4}{3} \right)^s<\delta(s).$$
\end{Prop}
\begin{proof}
Using the same argument as Proposition~\ref{Prop:Upper}, the proof is clear.
\end{proof}

\begin{Prop} \label{Prop:PrimeLower}
For $s \ge 4$, we have $$0<\frac{1}{P(s)}-2^s+\left( \frac{4}{3} \right)^s.$$
\end{Prop}
\begin{proof}
We will use proof by contradiction. Assume 
$$0 \ge \frac{1}{P(s)}-2^n+\left( \frac{4}{3} \right)^s .$$ 
We will first cancel some terms of our inequality. 
Our assumption implies that
$$\frac{6^s-4^s}{3^s} \ge \frac{1}{P(s)},$$
$$\frac{3^s}{6^s-4^s}  \le P(s)=\sum\limits_{m=1}^{\infty} \frac{1}{(p_m)^s},$$
$$3^s \le \sum\limits_{m=1}^{\infty} \frac{6^s-4^s}{(p_m)^s} =3^s-2^s +\sum\limits_{m=2}^{\infty} \frac{6^s-4^s}{(p_m)^s} ,$$
$$2^s \le \sum\limits_{m=2}^{\infty} \frac{6^s-4^s}{(p_m)^s},$$
$$1 \le \sum\limits_{m=2}^{\infty} \frac{3^s-2^s}{(p_m)^s}=1- \frac{2^s}{3^s} + \sum\limits_{m=3}^{\infty} \frac{3^s-2^s}{(p_m)^s} ,$$
$$ \frac{2^s}{9^s-6^s} \le \sum\limits_{m=3}^{\infty} \frac{1}{(p_m)^s},$$ 
$$\frac{2^s}{9^s-6^s}-\frac{1}{5^s} \le \frac{1}{7^s} +\frac{1}{11^s} + \frac{1}{13^s} + \cdots < \frac{1}{2} \left( \frac{1}{6^s} + \frac{1}{7^s} +\frac{1}{8^s} + \cdots \right).$$
It is clear that $$\frac{1}{6^s} + \frac{1}{7^s} +\frac{1}{8^s} + \cdots < \int_{5}^{\infty} \frac{1}{t^s} dt.$$
Therefore, $$\frac{2^s}{9^s-6^s}-\frac{1}{5^s}<\frac{1}{2} \int_{5}^{\infty} \frac{1}{t^s} dt=\frac{1}{2} \left( \frac{5^{1-s}}{s-1} \right).$$
Hence,
$$0>\frac{2^s}{9^s-6^s}-\frac{1}{5^s}-\frac{1}{2} \left( \frac{1}{(s-1)(5^{s-1})} \right)=$$
$$\frac{2^s}{9^s-6^s}-\frac{1}{5^s}-\left( \frac{5}{2s-2} \right) \left( \frac{1}{5^s} \right).$$
So
$$0>\frac{10^s}{9^s-6^s}-1-\frac{5}{2s-2}$$ and
$$\frac{2s+3}{2s-2}>\frac{10^s}{9^s-6^s} .$$
However, it is easy to see this fails for all $s \ge 4$.  Thus $$0<\frac{1}{P(s)}-2^n+\left( \frac{4}{3} \right)^s.$$
\end{proof}

\section{Proofs of the Results}

The following lemma is obvious.
\begin{Lem} \label{Lem:Floor}
If $$A<x<B,$$ then $$A-1<\lfloor{x}\rfloor<B.$$
\end{Lem}

\begin{Prop} \label{Prop:Floor}
For all natural numbers $n \ge 2$, we have that $$\left\lfloor \frac{1}{\zeta(n)-1} \right\rfloor=2^n-\left\lfloor \left( \frac{4}{3} \right)^n \right\rfloor-k,$$    where $k=1$ or $k=2$.
\end{Prop}
\begin{proof}
From Proposition~\ref{Prop:Lower} and Proposition~\ref{Prop:Upper} we know that $$1<\frac{1}{\zeta(n)-1}-2^n+\left( \frac{4}{3} \right)^n+2<1+\varepsilon(n).$$
Applying Lemma~\ref{Lem:Floor}  twice to this we obtain that $$-1<\left\lfloor\frac{1}{\zeta(n)-1}\right\rfloor - 2^n +\left\lfloor\left( \frac{4}{3} \right)^n\right\rfloor+2<1+\varepsilon(n).$$
 As $\lim\limits_{n\to\infty} \varepsilon(n)=0$, for $n \ge 5$  we have that $$-1<\left\lfloor\frac{1}{\zeta(n)-1}\right\rfloor - 2^n +\left\lfloor\left( \frac{4}{3} \right)^n\right\rfloor+2<2.$$
It is clear that $\left\lfloor\frac{1}{\zeta(n)-1}\right\rfloor - 2^n +\left\lfloor\left( \frac{4}{3} \right)^n\right\rfloor+2$ is an integer for $n \ge 2$.  Thus, $$\left\lfloor\frac{1}{\zeta(n)-1}\right\rfloor - 2^n +\left\lfloor\left( \frac{4}{3} \right)^n\right\rfloor+2=0 \textrm{ or } 1.$$  It is easy to check that this holds also for $2 \le n \le 4$.
Hence, $$\left\lfloor\frac{1}{\zeta(n)-1}\right\rfloor=2^n -\left\lfloor\left( \frac{4}{3} \right)^n\right\rfloor-k,$$ where $k=1$ or $k=2$ for all $n \ge 2$.
\end{proof}

\begin{Prop}\label{Prop:FractionalParts}
For all natural numbers, where $k$ is as in Proposition~\ref{Prop:Floor}, $n \ge 2$, we have that $$k-1< \left\{  \frac{1}{\zeta(n)-1} \right\} + \left\{ \left( \frac{4}{3} \right)^n \right\}<k-1+\varepsilon(n).$$
\end{Prop}
\begin{proof}
From Proposition~\ref{Prop:Floor} we have that $$\left\lfloor \frac{1}{\zeta(n)-1} \right\rfloor=2^n-\left\lfloor \left( \frac{4}{3} \right)^n \right\rfloor-k,$$ and therefore,
$$\left\lfloor \frac{1}{\zeta(n)-1} \right\rfloor-2^n+\left\lfloor \left( \frac{4}{3} \right)^n \right\rfloor+2=2-k.$$
Since the integral part plus the fractional part is the number, we have that
\begin{equation}
\label{eq:3}
\frac{1}{\zeta(n)-1}-2^n+ \left( \frac{4}{3} \right)^n +2=\left\{ \frac{1}{\zeta(n)-1} \right\} +\left\{ \left( \frac{4}{3} \right)^n \right\} +2-k.
\end{equation}
We know from Proposition~\ref{Prop:Lower} and Proposition~\ref{Prop:Upper} that $$1< \frac{1}{\zeta(n)-1} -2^n+\left( \frac{4}{3} \right)^n +2<1+\varepsilon(n)$$
Substituting \eqref{eq:3} into this we obtain that $$1<\left\{ \frac{1}{\zeta(n)-1} \right\} +\left\{ \left( \frac{4}{3} \right)^n \right\} +2-k<1+\varepsilon(n)$$
We conclude that $$k-1<\left\{ \frac{1}{\zeta(n)-1} \right\} +\left\{ \left( \frac{4}{3} \right)^n \right\}<k-1+\varepsilon(n).$$
\end{proof}

\begin{Prop}\label{Prop:k=1}
Let $k$ be as in Proposition~\ref{Prop:Floor}. Then $k=1$ finitely many times.
\end{Prop}
\begin{proof}
We know from Proposition~\ref{Prop:FractionalParts} that
$$k-1< \left\{  \frac{1}{\zeta(n)-1} \right\} + \left\{ \left( \frac{4}{3} \right)^n \right\}<k-1+\varepsilon(n).$$
It follows from Mahler's work, see \cite{Ma}, that if $p>q \ge 2$ are coprime integers, and $\varepsilon<0$, then
$$\left\{ \left(\frac{p}{q} \right)^n \right\}>e^{\varepsilon n} $$ for all integers $n$ except for at most a finite number of excpetions.

As $\varepsilon(n)=\mathcal{O}\left( \left(\frac{8}{9} \right)^n \right)$, by taking $p=4$, $q=3$ and $\varepsilon=\log(\frac{9}{10})$ we obtain that only a finite number of $n$ satisfy
$$0< \left\{ \left( \frac{4}{3} \right)^n \right\}<\varepsilon(n).$$
Therefore, only a finite number of $n$ satisfy
$$0< \left\{  \frac{1}{\zeta(n)-1} \right\} + \left\{ \left( \frac{4}{3} \right)^n \right\}<\varepsilon(n).$$
Thus, $k=1$ occurs a finite number of times.
\end{proof}

\begin{proof}[Proof of Theorem~\ref{Theo:Main}.]
Theorem~\ref{Theo:Main} follows from combining both Proposition~\ref{Prop:Floor} and Proposition~\ref{Prop:k=1}. 
\end{proof}

\begin{proof}[Proof of Corollary~\ref{Cor:FractionalParts}.]
Corollary~\ref{Cor:FractionalParts} follows from combining both Proposition~\ref{Prop:FractionalParts} and Proposition~\ref{Prop:k=1}.
\end{proof}

\begin{proof}[Proof of Corollary~\ref{Cor:Egypt}.]
We prove this by contradiction. Suppose $\zeta(n)=1 + \frac{1}{m}$ for some integer $m$, then
$\frac{1}{\zeta(n)-1} =m$ is an integer. Thus, $\left\{ \frac{1}{\zeta(n)-1} \right\}=0$. From Corollary~\ref{Cor:FractionalParts} we know that $$k-1<\left\{ \frac{1}{\zeta(n)-1} \right\} +\left\{ \left( \frac{4}{3} \right)^n \right\}<k-1+\varepsilon(n).$$
Therefore, $$k-1<\left\{ \left( \frac{4}{3} \right)^n \right\}<k-1+\varepsilon(n).$$
When $k=2$ we have $$1<\left\{ \left( \frac{4}{3} \right)^n \right\}<1+\varepsilon(n).$$
However, this implies that a fractional part is greater than $1$ which by definition is impossible. This is a contradiction therefore, if  $\zeta(n)=1 + \frac{1}{m}$ for some integer $m$, then $k=1$ which happens finitely many times.
\end{proof}
 
\begin{Prop} \label{Prop:Gen}
For all $\frac{1}{2}<x< \frac{3}{4}$ and $n$ large enough, $$-\left( \frac{4x}{3} \right)^n -x^n-k< \left \{ \frac{x^n}{\zeta(n)-1} \right \} -\left \{ \left(2x \right)^n \right \} < \varepsilon_x(n)-\left( \frac{4x}{3} \right)^n -x^n-k$$ where $k=0$ or $k=-1$ and when $x$ is rational
$k=0$ except for finite number of exceptions.
\end{Prop}
\begin{proof}
From Proposition~\ref{Prop:Lower} and Proposition~\ref{Prop:Upper} we know that 
$$1< \frac{1}{\zeta(n)-1} -2^n+\left( \frac{4}{3} \right)^n +2< 1+ \varepsilon(n).$$
So,
$$0< \frac{1}{\zeta(n)-1} -2^n+\left( \frac{4}{3} \right)^n +1< \varepsilon(n).$$
Recall that $\varepsilon_x(n)=(2x)^n \left( \left( \frac{2}{3} \right)^n + \left( \frac{1}{2} \right)^n \right)^2$. It is obvious that $x^n \varepsilon_y(n)= \varepsilon_{xy}(n)$. Therefore, 
$$0< \frac{x^n}{\zeta(n)-1} -\left(2x \right)^n +\left( \frac{4x}{3} \right)^n + x^n < \varepsilon_x(n)$$ and
\begin{equation}
\label{eq:4}
-\left( \frac{4x}{3} \right)^n-x^n< \frac{x^n}{\zeta(n)-1}-\left(2x \right)^n<\varepsilon_x(n)-\left( \frac{4x}{3} \right)^n-x^n.
\end{equation}
Using Lemma~\ref{Lem:Floor} twice and noticing that $\left\lfloor \left(2x \right)^n \right\rfloor$ appears with a negative sign we obtain that 
$$-1-\left( \frac{4x}{3} \right)^n-x^n<\left\lfloor \frac{x^n}{\zeta(n)-1} \right\rfloor-\left\lfloor \left(2x \right)^n \right\rfloor<1+\varepsilon_x(n)-\left( \frac{4x}{3} \right)^n-x^n.$$ 
Notice that for $\frac{1}{2}<x<\frac{3}{4}$ both bounds of the above inequality tend to $-1$ and $1$ respectively. Also notice that for $n$ large enough we have that 
$$\varepsilon_x(n)-\left( \frac{4x}{3} \right)^n-x^n=x^n \left (\varepsilon(n)-\left( \frac{4}{3} \right)^n-1 \right)<0.$$
As $\left\lfloor \frac{x^n}{\zeta(n)-1} \right\rfloor-\left\lfloor \left(2x \right)^n \right\rfloor$ is an integer for $n$ large enough we have that 
$$\left\lfloor \frac{x^n}{\zeta(n)-1} \right\rfloor-\left\lfloor \left(2x \right)^n \right\rfloor =-1 \textrm{ or } 0.$$
Therefore, 
$$\frac{x^n}{\zeta(n)-1} -\left(2x \right)^n=k+\left\{ \frac{x^n}{\zeta(n)-1} \right\}-\left\{ \left(2x \right)^n \right\},$$ where $k=-1$ or $k=0$.
Substituting this into \eqref{eq:4} we get that
$$-\left( \frac{4x}{3} \right)^n-x^n< k+\left\{ \frac{x^n}{\zeta(n)-1} \right\}-\left\{ \left(2x \right)^n \right\}<\varepsilon_x(n)-\left( \frac{4x}{3} \right)^n-x^n.$$
Hence, $$-\left( \frac{4x}{3} \right)^n-x^n-k< \left\{ \frac{x^n}{\zeta(n)-1} \right\}-\left\{ \left(2x \right)^n \right\}<\varepsilon_x(n)-\left( \frac{4x}{3} \right)^n-x^n-k.$$
One can use the same argument as in Proposition~\ref{Prop:k=1} to show that for a rational number $\frac{1}{2}<x<\frac{3}{4}$ we have that  $k=-1$ happens finitely many times.
\end{proof}

One can apply  the same methods for other values of $x$. For $0 < x \le \frac{1}{2}$, both fractional parts converge to $0$, so it is not interesting. For $x= \frac{3}{4}$, one can achieve the same result but with $k=-2$ or $k=-1$. For $\frac{3}{4}<x$, if we want the bounds to converge, there are two possibilities. The first is that there will be more than two fractional parts (apart for some exceptions like $x=1$), and $k$ will take more values, that is, the number of fractional parts. The second is that there will still be two fractional parts, and $k$ can only take two values, but one of the fractional parts will be of a finite sum of real numbers to natural powers rather than just one term. For example when $x= \frac{11}{10}$ we have that
$$-k< \left\{ \frac{ \left( \frac{11}{10} \right)^n}{\zeta(n)-1} \right\}-\left\{ \left( \frac{11}{5} \right)^n \right\}+\left\{ \left( \frac{22}{15} \right)^n \right\}+\left\{ \left( \frac{11}{10} \right)^n \right\}<-k+\varepsilon_{\frac{11}{10}}(n),$$ where $k=-2$ or $k=-1$ or $k=0$ or $k=1$. 
Or
$$-k<\left\{ \frac{ \left( \frac{11}{10} \right)^n}{\zeta(n)-1} \right\} - \left\{ \left( \frac{11}{5} \right)^n - \left( \frac{22}{15} \right)^n - \left( \frac{11}{10} \right)^n \right\} <-k+\varepsilon_{\frac{11}{10}}(n),$$ where $k=0$ or $k=1$.

\begin{proof}[Proof of Theorem ~\ref{Theo:Zeta}]
When $x= \frac{2}{3}$ in Proposition~\ref{Prop:Gen} we have
$$-\left( \frac{8}{9} \right)^n-\left( \frac{2}{3} \right)^n-k< \left\{ \frac{ \left( \frac{2}{3} \right)^n}{\zeta(n)-1} \right\}-\left\{ \left( \frac{4}{3} \right)^n \right\}<$$
$$\varepsilon_{\frac{2}{3}}(n)-\left( \frac{8}{9} \right)^n-\left( \frac{2}{3} \right)^n-k,$$
where $k=-1$ or $k=0$.
We know from Corollary~\ref{Cor:FractionalParts} that
$$k'-1< \left\{  \frac{1}{\zeta(n)-1} \right\} + \left\{ \left( \frac{4}{3} \right)^n \right\}<k'-1+\varepsilon(n),$$
where $k'=1$ or $k'=2$.  Adding the two inequalities we get
$$m-\left( \frac{8}{9} \right)^n-\left( \frac{2}{3} \right)^n<\left\{  \frac{1}{\zeta(n)-1} \right\} +\left\{ \frac{ \left( \frac{2}{3} \right)^n}{\zeta(n)-1} \right\}<$$
$$m+\varepsilon(n)+\varepsilon_{\frac{2}{3}}(n)-\left( \frac{8}{9} \right)^n-\left( \frac{2}{3} \right)^n,$$
where $m=0$ or $m=1$ or $m=2$. We can see that $m=0$ or $m=2$ occur finitely many times. It is obvious that $$0<\left\{  \frac{1}{\zeta(n)-1} \right\} +\left\{ \frac{ \left( \frac{2}{3} \right)^n}{\zeta(n)-1} \right\}<2.$$ Therefore, when $m=0$ the lower bound can be improved to $0$ and when $m=2$ the upper bound can be improved to $2$. This proves the theorem.
\end{proof}

\begin{proof}[Proof of Theorem ~\ref{Theo:Prime}]
We know from Proposition~\ref{Prop:PrimeUpper} and Proposition~\ref{Prop:PrimeLower} that $$0<\frac{1}{P(s)}-2^s+\left( \frac{4}{3} \right)^s< \delta(s),$$ for all $s \ge 4$.
We can also see from the proof of Proposition~\ref{Prop:Lower} that $$1<\frac{1}{\zeta(s)-1}-2^s+\left( \frac{4}{3} \right)^s+2,$$ for all $s \ge 7$.
In addition, from the proof of Proposition~\ref{Prop:Upper} it follows that $$\frac{1}{\zeta(s)-1}-2^s+\left( \frac{4}{3} \right)^s+2<1+\varepsilon(s),$$ for all $s>1$.
So for all $s \ge 7$ we have $$1<\frac{1}{\zeta(s)-1}-2^s+\left( \frac{4}{3} \right)^s+2<1+\varepsilon(s).$$
Taking the difference between the first inequality and the last inequality, and adding $2$ we obtain that $$1-\varepsilon(s)<\frac{1}{P(s)} - \frac{1}{\zeta(s)-1} < 1+\delta(s),$$ for all $s \ge 7$.
\end{proof}
Using similar arguments as previously more results can be derived about the fractional part of rational powers but this time related to the Prime Zeta function. This could also be done with other functions of infinite series of reciprocal powers, not just the Riemann Zeta function and the Prime Zeta function.


\begin{thebibliography}{99}


\bibitem{Ma} K.~Mahler, On the fractional parts of the powers of a rational number II, {\em Mathematika} {\bf 4} (1957), 122--124.


\bibitem{Vi} T.~Vijayaraghavan, On the fractional parts of the powers of a number I, ,{\em J. London Math. Soc.}, {\bf 15} (1940), 159--160.


\end{thebibliography}
\end{document}